\documentclass[11pt,a4paper]{article}

\usepackage[T1]{fontenc}
\usepackage[utf8]{inputenc}
\usepackage{lmodern}
\usepackage[margin=1in]{geometry}
\usepackage{amsmath,amssymb,amsfonts,amsthm,mathtools}
\usepackage{microtype}
\usepackage{enumitem}
\usepackage{xcolor}
\usepackage{hyperref}

\theoremstyle{plain}
\newtheorem{theorem}{Theorem}[section]
\newtheorem{proposition}[theorem]{Proposition}
\newtheorem{lemma}[theorem]{Lemma}
\newtheorem{corollary}[theorem]{Corollary}

\theoremstyle{definition}
\newtheorem{definition}[theorem]{Definition}

\newtheorem{question}[theorem]{Question}

\theoremstyle{remark}

\newcommand{\acyc}{\vec{\alpha}}
\newcommand{\dich}{\vec{\chi}}
\newcommand{\amin}{\vec a}
\newcommand{\tmax}{\vec t}
\newcommand{\Pp}{\mathbb{P}}
\newcommand{\Ee}{\mathbb{E}}
\newcommand{\NN}{\mathbb{N}}
\newcommand{\eps}{\varepsilon}
\newcommand{\fas}{\operatorname{fas}}

\newcommand{\hbin}{h_2}

\title{Feedback-arc robustness in random orientations of pseudorandom triangle-free graphs}

\author{Hui Lei\thanks{School of Statistics and Data Science, LPMC and KLMDASR,
    Nankai University. {\tt hlei@nankai.edu.cn}.
    Funded by the National Natural Science Foundation of China
    (Nos.\,12371351, 12431013),
    the Natural Science Foundation of Tianjin (24JCYBJC01670),
    and the Fundamental Research Funds for the Central Universities,
    Nankai University(63263098).}\and Danning Wang\thanks{School of Mathematics, Statistics and Mechanics, Beijing University of Technology, Beijing, China. Email: \texttt{wangdanning1113@163.com}.}
\and Yiqiao Wang\thanks{School of Mathematics, Statistics and Mechanics, Beijing University of Technology, Beijing, China. Email: \texttt{yqwang@bjut.edu.cn}. Funded by the National Natural Science Foundation of China
    (Nos.\,12422113).}}

\date{}

\begin{document}
\maketitle

\begin{abstract}
For an oriented graph $D$, let $\acyc(D)$ be the maximum order of an induced acyclic subdigraph, $\dich(D)$ its dichromatic number, and  $\fas(D)$  the minimum number of arcs whose deletion makes $D$ acyclic. We prove that for every fixed $\zeta\in(0,1/2)$, there are triangle-free graphs $G_n$ on $n$ vertices such that a uniformly random orientation $D_n$ satisfies,
\[
 \left(\frac12-\zeta\right)e(G_n[U])
 <\fas(D_n[U])\le \frac12e(G_n[U])
\] with probability at least $1-\exp\!\left[-\Omega_\zeta\!\left(\sqrt n\,(\log n)^{3/2}\right)\right]$ simultaneously for every vertex set $U$ of size at least $C_\zeta\sqrt{n\log n}$. The upper bound is universal, so the feedback-arc ratio can be made arbitrarily close to the largest possible value, uniformly over all sufficiently large induced subdigraphs.

In particular,
$\acyc(D_n)=O(\sqrt{n\log n})$,
and every linear-size induced subdigraph has dichromatic number $\Omega(\sqrt{n/\log n})$. This yields
$\amin(n)=\Theta(\sqrt{n\log n})$ and $\tmax(n)=\Theta\!\left(\sqrt{\frac n{\log n}}\right)$, where $\amin(n)$ and $\tmax(n)$  denote, respectively,  the
minimum of $\acyc(D)$ and the maximum of $\dich(D)$ over all oriented triangle-free graphs $D$ of
order $n$.
This confirms two conjectures of Aboulker, Havet, Pirot, and Schabanel.

We also prove a robust oriented version of the Guo--Warnke packing theorem. Every sufficiently lower-uniform host graph admits an approximate edge-decomposition into $\Theta(\sqrt{n/\log n})$ spanning triangle-free graphs whose independent uniform orientations simultaneously satisfy the same near-maximal feedback-arc property in all sufficiently large induced subgraphs. Consequently, almost all edges of $K_n$, and of every fixed-density binomial random graph, can be packed into edge-disjoint oriented triangle-free graphs with hereditary feedback-arc robustness and vertex-robust asymptotically optimal dichromatic number.

\end{abstract}

\section{Introduction}

An \emph{oriented graph} is an orientation of a finite simple graph. A vertex set in an oriented graph $D$ is \emph{acyclic} if it induces no directed cycle. The \emph{acyclic number} $\acyc(D)$ is the maximum size of an acyclic vertex set, and the \emph{dichromatic number} $\dich(D)$ is the minimum number of acyclic sets in a partition of $V(D)$; see Neumann-Lara~\cite{NeumannLara}. Obviously, we have $\dich(D)\ge \frac{|V(D)|}{\acyc(D)}$.
The underlying graph of $D$ is denoted by $U(D)$. We call $D$ \emph{triangle-free} when $U(D)$ is triangle-free.

We also measure how far an orientation is from acyclic. The \emph{feedback arc set number} $\fas(D)$ is the minimum number of arcs whose deletion makes $D$ acyclic. Equivalently, $\fas(D)$ is the minimum number of arcs that must be reversed to obtain an acyclic orientation of $U(D)$; see Proposition~\ref{prop:fas-orders}. Every orientation with $m$ arcs satisfies $\fas(D)\le m/2$, so a lower bound of the form $\fas(D)\ge(1/2-o(1))m$ is asymptotically best possible.

For $n\in\NN$, Aboulker, Havet, Pirot, and Schabanel~\cite{AboulkerHPS} introduced the extremal functions
\[
 \amin(n):=\min\{\acyc(D):D\text{ is an $n$-vertex oriented triangle-free graph}\}
\]
and
\[
 \tmax(n):=\max\{\dich(D):D\text{ is an $n$-vertex oriented triangle-free graph}\}.
\]
Every independent set of $U(D)$ is acyclic in $D$. Hence the inverse Ramsey bound of Ajtai, Koml\'os, and Szemer\'edi~\cite{AKS}, sharpened by Shearer~\cite{Shearer}, gives $\amin(n)=\Omega(\sqrt{n\log n})$. The complementary Ramsey constructions were developed through the work of Kim~\cite{Kim} and the triangle-free process analyses of Fiz Pontiveros, Griffiths, and Morris~\cite{FizGriffithsMorris} and Bohman and Keevash~\cite{BohmanKeevash}. Aboulker et al.~\cite{AboulkerHPS} proved $\amin(n)=O(\sqrt n\log n)$ and conjectured that the remaining factor $\sqrt{\log n}$ can be removed. They also made the corresponding conjecture for $\tmax(n)$.

Their orientation argument starts from the observation that every $n$-vertex graph has an orientation in which each acyclic set spans average degree at most $2\log_2 n+1$; see~\cite[Lemma~2.4]{AboulkerHPS}. For the pseudorandom triangle-free Ramsey graphs of Guo and Warnke~\cite{GuoWarnke}, our main theorem gives a much stronger conclusion: not only is every large set non-acyclic, but every large induced orientation has feedback arc set number within an arbitrarily small multiplicative error of half the total number of arcs.

\begin{theorem}\label{thm:single}
For every $\zeta\in(0,1/2)$ there are constants $C,c>0$ such that, for every sufficiently large $n$, there are an $n$-vertex triangle-free graph $G_n$ and an integer $k_n\le C\sqrt{n\log n}$
with the following property. If $D_n$ is a uniformly random orientation of $G_n$, then
\[ \Pp\!\left(
   \exists U\subseteq V(G_n),\ |U|\ge k_n:\
   \fas(D_n[U])\le\left(\frac12-\zeta\right)e(G_n[U])
 \right)
 \le \exp\!\left[-c\sqrt n\,(\log n)^{3/2}\right].\]
On the complementary event,
\begin{enumerate}[label=\textnormal{(\roman*)}]
\item $\acyc(D_n)<k_n$;
\item every $W\subseteq V(G_n)$ with $|W|\ge k_n$ satisfies
$\dich(D_n[W])\ge \frac{|W|}{k_n-1}$;
\item one needs to perform at least $c\sqrt n\,(\log n)^{3/2}$ arc edits---each edit deleting or reversing one arc---before an acyclic set of order $k_n$ can be created. In particular, every digraph obtained using fewer edits still has acyclic number less than $k_n$ and dichromatic number at least $n/(k_n-1)$.
\end{enumerate}
\end{theorem}

The last assertion is a global arc-edit robustness statement. In particular, after deleting any fixed positive proportion of the vertices, the remaining oriented graph still has dichromatic number of order $\sqrt{n/\log n}$.

Theorem~\ref{thm:single} immediately settles the two extremal conjectures.

\begin{theorem}\label{thm:extremal}
There are absolute constants $c,C>0$ such that, for all sufficiently large $n$,
\[
 c\sqrt{n\log n}\le \amin(n)\le C\sqrt{n\log n}
\]
and
\[
 c\sqrt{\frac n{\log n}}\le \tmax(n)\le C\sqrt{\frac n{\log n}}.
\]
Equivalently,
\[
    \amin(n)=\Theta(\sqrt{n\log n})
    \quad\text{and}\quad
    \tmax(n)=\Theta\!\left(\sqrt{\frac n{\log n}}\right).
\]
\end{theorem}

The graph $G_n$ comes from the pseudorandom triangle-free subgraph theorem of Guo and Warnke~\cite{GuoWarnke}. Their theorem constructs, inside a dense host graph, a triangle-free graph whose edge distribution between all sufficiently large disjoint sets resembles that of a binomial random subgraph. The orientation step is an entropy argument. For a fixed ordering of the vertices, the number of backward arcs in a uniform random orientation follows a binomial distribution $\operatorname{Bin}(m,1/2)$. A union bound over vertex sets and orderings proves that every large induced subdigraph has nearly $m/2$ backward arcs under every ordering.

The same argument preserves the packing feature of the Guo--Warnke construction. For disjoint vertex sets $A,B$ in a graph $H$, let $e_H(A,B)$ denote the number of edges with one endpoint in each set.

\begin{definition}\label{def:lower-uniform}
For $\xi>0$ and an integer $s\ge1$, an $n$-vertex graph $H$ is \emph{$(\xi,s)$-lower-uniform} if
$e_H(A,B)\ge \xi |A||B|$
for every pair of disjoint sets $A,B\subseteq V(H)$ with $|A|=|B|=s$.
\end{definition}

\begin{theorem}\label{thm:packing}
For every $\eps,\xi,C_0>0$ and $\zeta\in(0,1/2)$ there are constants $C,R,c>0$ and $n_0$ such that the following holds for every $n\ge n_0$. Let $t=\left\lceil C_0\sqrt{n\log n}\right\rceil$,
and let $H$ be an $n$-vertex $(\xi,t)$-lower-uniform graph. Then $H$ contains an edge-disjoint collection
$G_0,\ldots,G_{r-1}$
of spanning triangle-free graphs, where
$r=\left\lceil R\sqrt{\frac n{\log n}}\right\rceil$,
such that
\[\sum_{i=0}^{r-1}e(G_i)\ge(1-\eps)e(H).\]
There is an integer $k_n\le C\sqrt{n\log n}$ such that, if the edges of the $G_i$ are oriented independently and uniformly, producing $D_i$, then
\[
\begin{aligned}
&\Pp\!\left(\exists i<r,\ \exists U\subseteq V(H),\ |U|\ge k_n:\ \fas(D_i[U])\le\left(\tfrac12-\zeta\right)e(G_i[U])\right)\\
&\quad\le \exp\!\left[-c\sqrt n\,(\log n)^{3/2}\right].
\end{aligned}
\]
Consequently, outside this exceptional event, every $i<r$ and every $W\subseteq V(H)$ with $|W|\ge k_n$ satisfy
$\dich(D_i[W])\ge \frac{|W|}{k_n-1}$.
\end{theorem}

Taking $H=K_n$ (the complete graph on
$n$ vertices) gives an approximate decomposition into oriented Ramsey graphs that are hereditarily far from acyclic.

\begin{corollary}\label{cor:complete-packing}
For every $\eps>0$ and $\zeta\in(0,1/2)$, all sufficiently large complete graphs $K_n$ contain $\Theta_{\eps,\zeta}(\sqrt{n/\log n})$ edge-disjoint spanning triangle-free graphs covering all but at most $\eps\binom n2$ edges, such that independent uniform orientations simultaneously satisfy the feedback-arc and hereditary dichromatic conclusions of Theorem~\ref{thm:packing}.
\end{corollary}

The same conclusion holds in fixed-density binomial random graphs.

\begin{corollary}\label{cor:random-host}
Fix $p\in(0,1]$, $\eps>0$, and $\zeta\in(0,1/2)$. With probability $1-o(1)$, a graph $H\sim G(n,p)$ contains an edge-disjoint family of $\Theta_{p,\eps,\zeta}(\sqrt{n/\log n})$ spanning triangle-free graphs covering at least $(1-\eps)e(H)$ edges, for which independent uniform orientations satisfy the conclusion of Theorem~\ref{thm:packing} with conditional probability at least $1-\exp\!\left[-\Omega_{p,\eps,\zeta}\!\left(\sqrt n\,(\log n)^{3/2}\right)\right]$.
Moreover, $e(H)=(1+o(1))p\binom n2$ with probability $1-o(1)$.
\end{corollary}

Apart from the published Guo--Warnke pseudorandom-subgraph theorem and the standard inverse Ramsey estimate quoted below, all arguments needed for our results are included in full. All logarithms are natural unless a base is displayed. We make no attempt to optimise constants.

\section{Feedback arc sets and an orientation transfer theorem}

\subsection{Orders, entropy, and feedback arc sets}

Let $D$ be an orientation of a graph $G$. An \emph{ordering} $\sigma$ of $D$ is an enumeration $v_1,v_2,\dots$ of $V(G)$. An arc $v_iv_j$ is \emph{backward} under $\sigma$ if $i>j$. Let $b_\sigma(D)$ be the number of backward arcs.
We first recall the following elementary fact.

\begin{proposition}\label{prop:fas-orders}
For every oriented graph $D$, $\fas(D)=\min_{\sigma} b_\sigma(D)$, where the minimum is over all orderings of $V(D)$. Consequently, $\fas(D)\le e(U(D))/2$. The same minimum is the least number of arcs that must be reversed to make $D$ acyclic.
\end{proposition}

\begin{proof}
For any ordering $\sigma$, deleting all backward arcs leaves an orientation in which every remaining arc is forward. Such an orientation is acyclic, so $\fas(D)\le b_\sigma(D)$.
Taking the minimum over $\sigma$ gives $\fas(D)\le\min_\sigma b_\sigma(D)$.

Conversely, let $E_0$ be a minimum feedback arc set, and take a topological ordering $\sigma$ of the acyclic digraph $D-E_0$, so that all its arcs are forward. Every backward arc of $D$ under $\sigma$ then belongs to $E_0$. Hence
\[
    \min_\sigma b_\sigma(D)\le b_\sigma(D)\le |E_0|=\fas(D),
\]
which is the reverse inequality, proving the identity.

If $\sigma$ is a uniformly random ordering, each arc is backward with probability $1/2$. Therefore $\Ee b_\sigma(D)=e(U(D))/2$, and some ordering has at most this many backward arcs, proving $\fas(D)\le e(U(D))/2$.

Finally, reversing the backward arcs of an ordering makes every arc forward and therefore gives an acyclic orientation. Conversely, if reversing a set $R$ of arcs produces an acyclic orientation, then in a topological ordering of the resulting orientation every original backward arc lies in $R$. Thus the minimum reversal number is also $\min_\sigma b_\sigma(D)$.
\end{proof}

For $x\in[0,1]$, let
\[
    \hbin(x):=-x\log_2x-(1-x)\log_2(1-x),
\]
with the usual convention $0\log_2 0=0$.

\begin{lemma}\label{lem:entropy}
For every integer $m\ge1$ and every $\theta\in[0,1/2]$,
\[
    \sum_{j=0}^{\lfloor\theta m\rfloor}\binom mj
    \le 2^{\hbin(\theta)m}.
\]
Consequently, if $X\sim\operatorname{Bin}(m,1/2)$, then $\Pp(X\le\theta m)
    \le 2^{-[1-\hbin(\theta)]m}$.
\end{lemma}

\begin{proof}
The assertion is immediate for $\theta=0$. Suppose $0<\theta\le1/2$ and let $x=\theta/(1-\theta)\le1$. For every $j\le\theta m$ we have $x^j\ge x^{\theta m}$. Hence
\[
 (1+x)^m
 =\sum_{j=0}^m\binom mjx^j
 \ge x^{\theta m}\sum_{j=0}^{\lfloor\theta m\rfloor}\binom mj.
\]
Therefore
\[
 \sum_{j=0}^{\lfloor\theta m\rfloor}\binom mj
 \le\left(\frac{1+x}{x^\theta}\right)^m
 =2^{\hbin(\theta)m},
\]
which is the first bound. Dividing by $2^m$ gives the second.
\end{proof}

\begin{lemma}\label{lem:fixed-orientation}
Let $G$ be a graph on $n$ vertices with $m$ edges, and let $D$ be a uniformly random orientation of $G$. For every $\theta\in[0,1/2]$,
\[
  \Pp\bigl(\fas(D
)\le\theta m\bigr)
  \le n!\,2^{-[1-\hbin(\theta)]m}.
\]
In particular, for $\theta=0$,
\[
    \Pp(D\text{ is acyclic})\le n!\,2^{-m}.
\]
\end{lemma}

\begin{proof}
Fix an ordering $\sigma$ of $V(G)$. Under a uniformly random orientation, the indicators that the edges are backward under $\sigma$ are independent Bernoulli random variables with parameter $1/2$. Thus $b_\sigma(D)\sim\operatorname{Bin}(m,1/2)$. By Proposition~\ref{prop:fas-orders}, the event $\fas(D)\le\theta m$ implies that some one of the $n!$ orderings has at most $\theta m$ backward arcs. Lemma~\ref{lem:entropy} and a union bound over the orderings give the stated bound.
\end{proof}

\subsection{From cross-density to uniform feedback-arc robustness}

The following averaging lemma upgrades lower bounds between equal-sized disjoint sets to induced-density bounds at every larger scale.

\begin{lemma}\label{lem:cross-to-induced}
Let $G$ be a graph. Suppose that for some integer $s\ge1$ and some $q>0$, $e_G(A,B)\ge qs^2$ for every pair of disjoint $s$-vertex sets $A,B\subseteq V(G)$. Then every set $U\subseteq V(G)$ of size $u\ge2s$ satisfies $e(G[U])\ge q\binom u2$.
\end{lemma}

\begin{proof}
Choose an ordered pair $(A,B)$ uniformly from all pairs of disjoint $s$-vertex subsets of $U$. Each edge of $G[U]$ has probability
$\frac{2s^2}{u(u-1)}$ of having one endpoint in $A$ and the other in $B$. Therefore
\[
    \Ee e_G(A,B)
    =\frac{2s^2}{u(u-1)}e(G[U]).
\]
The hypothesis gives $\Ee e_G(A,B)\ge qs^2$. Cancelling $s^2$ and rearranging yields $e(G[U])\ge q\binom u2$.
\end{proof}

We isolate the probabilistic transfer statement used throughout.

\begin{theorem}\label{thm:transfer}
Let $G$ be an $n$-vertex graph. Suppose that for some integer $s\ge1$ and some $q>0$, $e_G(A,B)\ge qs^2$
for every pair of disjoint $s$-vertex sets $A,B$. Let $k=2s$, fix $\zeta\in(0,1/2)$, and define $\kappa_\zeta:=1-\hbin(1/2-\zeta)>0$. If $\kappa_\zeta q(k-1)\ge8\log_2 n$, then a uniformly random orientation $D$ of $G$ satisfies
\[
 \Pp\!\left(\exists U\subseteq V(G),\ |U|\ge k:\ \fas(D[U])\le\left(\tfrac12-\zeta\right)e(G[U])\right)\le 2n^{-3k}.
\]
Consequently, outside this exceptional event, every induced subdigraph on at least $k$ vertices has acyclic number less than $k$ and feedback arc set number greater than $(1/2-\zeta)$ times its number of arcs.
\end{theorem}

\begin{proof}
Write $\theta=1/2-\zeta$ and $\kappa=1-\hbin(\theta)$. Fix $u\ge k$ and a set $U$ of size $u$. By Lemma~\ref{lem:cross-to-induced}, $e(G[U])\ge q\binom u2$.
Applying Lemma~\ref{lem:fixed-orientation} to $D[U]$ gives
\[
 \Pp\bigl(\fas(D[U])\le\theta e(G[U])\bigr)
 \le u!\,2^{-\kappa q\binom u2}.
\]
Taking a union bound over the $\binom nu$ choices of $U$ and using $\binom nu u!=(n)_u\le n^u$, we obtain
\[
 \Pp\!\left(\exists U,\ |U|=u:\ \fas(D[U])\le\theta e(G[U])\right)\le n^u2^{-\kappa q u(u-1)/2}.
\]
Since $u\ge k$, the hypothesis $\kappa q(k-1)\ge8\log_2 n$ implies
$\frac{\kappa q(u-1)}2
    \ge4\log_2 n$.
Thus $n^u2^{-\kappa q u(u-1)/2}\le n^{-3u}$. Summing over $u=k,\ldots,n$ gives
\[
    \sum_{u=k}^n n^{-3u}
    \le\frac{n^{-3k}}{1-n^{-3}}
    \le2n^{-3k}
\]
for $n\ge2$, proving the stated bound.

If an induced subdigraph on at least $k$ vertices contained an acyclic set of order $k$, the induced orientation on that $k$-set would have feedback arc set number zero, contradicting the good event. This proves the final assertion.
\end{proof}

\section{Pseudorandom triangle-free subgraphs}

\subsection{The Guo--Warnke input}

We use the following theorem as a black box. It is Theorem~4 of Guo and Warnke~\cite{GuoWarnke}.

\begin{theorem}[\cite{GuoWarnke}]\label{thm:GW}
There are absolute constants $\beta_0,R_0>0$ such that, for all $\gamma,\delta\in(0,1]$, $\beta\in(0,\beta_0)$, and
$L\ge \frac{R_0}{\delta^2\sqrt{\beta\gamma}}$, the following holds for all sufficiently large $n$ with
$\rho=\sqrt{\frac{\beta\log n}{n}}$. For any $n$-vertex graph $H$, there exists a triangle-free spanning subgraph $G$  such that
\[
    e_G(A,B)=(1\pm\delta)\rho e_H(A,B)
\]
for all disjoint $A,B\subseteq V(H)$ satisfying
$|A|=|B|=s$ with $s=\left\lceil L\sqrt{n\log n}\right\rceil$
 and $e_H(A,B)\ge\gamma|A||B|$.
Here $x=(1\pm\delta)y$ means $(1-\delta)y\le x\le(1+\delta)y$.
\end{theorem}

Crucially, $L$ may be chosen arbitrarily large after the other parameters are fixed, which lets the entropy exponent in Theorem~\ref{thm:transfer} dominate the number of vertex sets and orderings.

\subsection{Two double-counting lemmas}

\begin{lemma}\label{lem:scale-lift}
Let $1\le t\le s$ and let $H$ be $(\xi,t)$-lower-uniform. Then $H$ is $(\xi,s)$-lower-uniform.
\end{lemma}

\begin{proof}
Fix disjoint $s$-vertex sets $A,B$. Sum $e_H(A',B')$ over all pairs $A'\in\binom At$ and $B'\in\binom Bt$. By lower-uniformity, this sum is at least $\binom st^2\xi t^2$.
On the other hand, every edge of $H[A,B]$ occurs in exactly $\binom{s-1}{t-1}^2$ summands. Therefore
\[
 e_H(A,B)\binom{s-1}{t-1}^2
 \ge \binom st^2\xi t^2.
\]
Since $t\binom st=s\binom{s-1}{t-1}$, division gives $e_H(A,B)\ge\xi s^2$.
\end{proof}

\begin{lemma}\label{lem:edge-double-count}
Let $F$ be an
$n$-vertex spanning subgraph of  $H$, and suppose $2s\le n$. If $e_F(A,B)\ge(1-\eps)e_H(A,B)$ for every ordered pair $(A,B)$ of disjoint $s$-vertex sets, then
$e(F)\ge(1-\eps)e(H)$.
\end{lemma}

\begin{proof}
Sum the assumed inequality over all ordered pairs $(A,B)$ of disjoint $s$-vertex sets. Every edge $xy$ of either graph is counted the same number
$2\binom{n-2}{s-1}\binom{n-s-1}{s-1}$ of times: choose which endpoint lies in $A$, choose the other $s-1$ vertices of $A$, and then choose the other $s-1$ vertices of $B$. Cancelling this positive common factor proves the assertion.
\end{proof}

\section{One robustly oriented triangle-free graph}

\begin{proof}[Proof of Theorem~\ref{thm:single}]
Fix $\zeta\in(0,1/2)$ and let
$\kappa=1-\hbin\!\left(\frac12-\zeta\right)>0$.
Let $\beta_0,R_0$ be the constants from Theorem~\ref{thm:GW}. Fix
$\delta=\frac14$ and $\beta=\frac{\beta_0}{2}$.
Choose a constant $L$ satisfying
$L\ge \min\left\{\frac{R_0}{\delta^2\sqrt\beta},  \frac{8}{\kappa(1-\delta)\sqrt\beta\,\log2}\right\}$.
For sufficiently large $n$, let
$\rho=\sqrt{\frac{\beta\log n}{n}}$, $s=\left\lceil L\sqrt{n\log n}\right\rceil$, and $k=2s$.
By increasing the lower bound on $n$ if necessary, assume $k\le n$.

Apply Theorem~\ref{thm:GW} to $H=K_n$ with $\gamma=1$. We obtain a spanning triangle-free graph $G_n$ such that
$e_{G_n}(A,B)\ge(1-\delta)\rho s^2$
for every pair of disjoint $s$-vertex sets $A,B$. Let $q=(1-\delta)\rho$.
Since $k-1=2s-1\ge s$, the second condition on $L$ gives
\begin{align*}
    \kappa q(k-1)
    &\ge \kappa(1-\delta)\rho s\\
    &\ge \kappa(1-\delta)L\sqrt\beta\,\log n\\
    &\ge8\log_2 n.
\end{align*}
Theorem~\ref{thm:transfer} therefore yields
\[
 \Pp\!\left(\exists U,\ |U|\ge k:\ \fas(D_n[U])\le\left(\tfrac12-\zeta\right)e(G_n[U])\right)\le2n^{-3k}.
\]
For large $n$,
$k=2\left\lceil L\sqrt{n\log n}\right\rceil
    \le3L\sqrt{n\log n}$.
Moreover, $2n^{-3k}
    \le\exp\!\left[-c_1\sqrt n\,(\log n)^{3/2}\right]$
for some $c_1=c_1(\zeta)>0$. Taking $k_n=k$ and $C=3L$ proves the probability bound of Theorem~\ref{thm:single}.

On the good event, there is no acyclic set of order $k$, because such a set would induce feedback arc set number zero. Thus $\acyc(D_n)<k$. More generally, if $W\subseteq V(G_n)$ has $|W|\ge k$, then every acyclic set in $D_n[W]$ has order at most $k-1$. Any dicolouring of $D_n[W]$ therefore uses at least $|W|/(k-1)$ colours, which proves~(ii).

It remains to prove the global arc-edit assertion. By Lemma~\ref{lem:cross-to-induced}, every $k$-vertex set $U$ satisfies
\[
    e(G_n[U])\ge q\binom k2\ge q s^2.
\]
Since $s\ge L\sqrt{n\log n}$,
\begin{align*}
    q s^2
    &=(1-\delta)\sqrt{\frac{\beta\log n}{n}}\,s^2\\
    &\ge(1-\delta)L^2\sqrt\beta\,\sqrt n\,(\log n)^{3/2}.
\end{align*}
Let $\theta=1/2-\zeta$. On the good event, making $D_n[U]$ acyclic by deleting arcs requires more than $\theta e(G_n[U])$ deletions for every $k$-set $U$; by Proposition~\ref{prop:fas-orders}, the same lower bound holds for reversals. Hence, after decreasing the constant if necessary, fewer than
$c\sqrt n\,(\log n)^{3/2}$
arc edits, each of which deletes or reverses one arc, cannot create an acyclic $k$-set. Indeed, suppose that the edited digraph contained an acyclic $k$-set $U$. Delete from the original $D_n[U]$ every arc that was edited inside $U$. The remaining digraph is a subdigraph of the edited acyclic digraph on $U$, and is therefore acyclic. Thus the number of edits inside $U$ is at least $\fas(D_n[U])>\theta e(G_n[U])$, contradicting the assumed global edit budget. Consequently the edited digraph has acyclic number less than $k$, and hence dichromatic number at least $n/(k-1)$. This proves~(iii).
\end{proof}

For completeness, we record how the standard inverse Ramsey bound gives the other directions in Theorem~\ref{thm:extremal}.
We first recall Shearer's bound on independent sets in triangle-free graphs.

\begin{theorem}[\cite{Shearer}]\label{thm:inverse-ramsey}
There are constants $c_0>0$ and $n_0$ such that every triangle-free graph on $n\ge n_0$ vertices has an independent set of size at least $c_0\sqrt{n\log n}$.
\end{theorem}

Using Theorem~\ref{thm:inverse-ramsey}, we now establish the following bound on the oriented chromatic number.

\begin{lemma}\label{lem:delete-independent}
 There is an absolute constant $C_0$ such that every oriented triangle-free graph $D$ on $n$ vertices satisfies
$\dich(D)\le C_0\sqrt{\frac n{\log n}}$.
\end{lemma}

\begin{proof}
Every independent set in $U(D)$ is acyclic in $D$. Repeatedly remove an independent set of maximum size from the underlying graph of the remaining digraph, assigning a new colour to each removed set.

Enlarge $n_0$ if necessary so that $n_0\ge \mathrm e^2$, and let $J=\lfloor\log_2(n/n_0)\rfloor$. For $0\le j\le J-1$, consider the stage during which the number $m$ of remaining vertices lies in
$\left(\frac{n}{2^{j+1}},\frac{n}{2^j}\right]$.
At every removal in this stage, Theorem~\ref{thm:inverse-ramsey} supplies an independent set of size at least
$c_0\sqrt{\frac{n}{2^{j+1}}\log\!\left(\frac{n}{2^{j+1}}\right)}$.
Consequently, using $n/2^j$ as an upper bound on the total number of vertices removed during the stage, the number of colours used there is at most $\frac{\sqrt{2}}{c_0}\sqrt{\frac{n/2^j}{\log(n/2^{j+1})}}+1$.
After these stages, fewer than $n/2^J<2n_0$ vertices remain, costing at most $2n_0$ additional colours.

Let $J_1=\lfloor(\log_2 n)/2\rfloor-2$. For $j\le J_1$, we have $\log(n/2^{j+1})\ge\frac13\log n$ for all sufficiently large $n$. Summing the per-stage bound over these $j$ gives
\[
 O\!\left(\sqrt{\frac n{\log n}}
          \sum_{j\ge0}2^{-j/2}\right)
 =O\!\left(\sqrt{\frac n{\log n}}\right).
\]
For $j>J_1$, the number of remaining vertices is at most $O(\sqrt n)$. There are $O(\log n)$ such stages, and each per-stage term is $O(n^{1/4}+1)$. Their total is
$O(n^{1/4}\log n)=o\!\left(\sqrt{\frac n{\log n}}\right)$.
This proves the lemma.
\end{proof}

\begin{proof}[Proof of Theorem~\ref{thm:extremal}]
The lower bound on $\amin(n)$ follows because every independent set in the underlying triangle-free graph is acyclic, followed by Theorem~\ref{thm:inverse-ramsey}. The upper bound on $\amin(n)$ follows from Theorem~\ref{thm:single}.

Lemma~\ref{lem:delete-independent} gives the upper bound on $\tmax(n)$. For the lower bound, choose an orientation $D_n$ supplied by Theorem~\ref{thm:single} and apply the bound $\dich(D_n)\ge n/\acyc(D_n)$:
\[
    \tmax(n)\ge\dich(D_n)
    \ge\frac{n}{\acyc(D_n)}
    =\Omega\!\left(\sqrt{\frac n{\log n}}\right).
\]
\end{proof}

The hereditary statement in Theorem~\ref{thm:single} has the following useful vertex-robustness consequence.

\begin{corollary}\label{cor:vertex-robustness}
Fix $\zeta\in(0,1/2)$ and $\lambda>0$. The graphs $G_n$ in Theorem~\ref{thm:single} have the following property with probability tending to one under uniform random orientation: for every $W\subseteq V(G_n)$ with $|W|\ge\lambda n$,
$\dich(D_n[W])\ge c_{\lambda,\zeta}\sqrt{\frac n{\log n}}$.
In particular, deleting any set of at most $(1-\lambda)n$ vertices preserves dichromatic number of the optimal order $\sqrt{n/\log n}$.
\end{corollary}

\begin{proof}
For sufficiently large $n$, $\lambda n\ge k_n$. Apply the hereditary bound~(ii) of Theorem~\ref{thm:single} and use $k_n\le C_\zeta\sqrt{n\log n}$.
\end{proof}

\section{Packing robustly oriented triangle-free graphs}

The construction and coverage argument follow Guo and Warnke~\cite[Theorem~5]{GuoWarnke}; the new point is that every packed graph, after random orientation, is simultaneously almost maximally far from acyclic in every sufficiently large induced subdigraph.

\begin{proof}[Proof of Theorem~\ref{thm:packing}]
If $\xi>1$, then the hypothesis is vacuous, so assume $0<\xi\le1$. Replacing $\eps$ by $\min\{\eps,1/2\}$ only strengthens the conclusion, so assume $0<\eps\le1/2$.
Let $\kappa=1-\hbin\!\left(\frac12-\zeta\right)>0$ and $\beta_0,R_0$ be the constants in Theorem~\ref{thm:GW}. Fix
$\delta=\tfrac14, \beta=\tfrac{\beta_0}{2}$ and $\gamma=\eps^2\xi$.
Choose a constant $L$ such that $L\ge \min\left\{C_0, \frac{R_0}{\delta^2\sqrt{\beta\gamma}},
 \frac{8}{\kappa(1-\delta)\gamma\sqrt\beta\,\log2}\right\}$.
For all sufficiently large $n$, define $\rho=\sqrt{\tfrac{\beta\log n}{n}}, s=\left\lceil L\sqrt{n\log n}\right\rceil$ and $k=2s$.
Increase $n_0$ if necessary so that $k\le n$, and let
$I=\left\lceil\frac{\log(1/\eps)}{(1-\delta)\rho}\right\rceil$.
Thus
 $I=\left\lceil R\sqrt{\frac n{\log n}}\right\rceil$ and
 $R:=\frac{\log(1/\eps)}{(1-\delta)\sqrt\beta}$.
We construct $I$ graphs and then let $r=I$.

Let $t=\left\lceil C_0\sqrt{n\log n}\right\rceil$.
Since $s\ge t$, Lemma~\ref{lem:scale-lift} implies that $H$ is $(\xi,s)$-lower-uniform. Let $H_0=H$. We inductively construct spanning triangle-free graphs
$G_i\subseteq H_i$ and  $H_{i+1}=H_i\setminus G_i$ where $0\le i<I$,
so that, for every pair of disjoint $s$-vertex sets $A,B$,
\[
 \bigl(1-(1+\delta)\rho\bigr)^i e_H(A,B)
 \le e_{H_i}(A,B)
 \le \bigl(1-(1-\delta)\rho\bigr)^i e_H(A,B).
\]
The assertion is trivial for $i=0$.

Suppose $0\le i<I$ and $H_i$ has been constructed. Since $\rho\to0$, for all sufficiently large $n$ we have $1-(1+\delta)\rho\ge \exp(-(1+2\delta)\rho)$. Indeed, the function $\log(1-(1+\delta)x)+(1+2\delta)x$ is positive for all sufficiently small positive $x$, because it vanishes at $0$ and has derivative $\delta$ there. Since $i\le I-1$, the definition of $I$ gives
 $i<\frac{\log(1/\eps)}{(1-\delta)\rho}$.
Using the invariant, the preceding inequality, the identity
$\frac{1+2\delta}{1-\delta}=2$ for $\delta=\frac14$,
and the lower-uniformity of $H$, we obtain
\begin{align*}
 e_{H_i}(A,B)
 &\ge \bigl(1-(1+\delta)\rho\bigr)^i e_H(A,B)\\
 &\ge \exp\bigl(-(1+2\delta)\rho i\bigr)e_H(A,B)\\
 &\ge \eps^2 e_H(A,B)\\
 &\ge \eps^2\xi s^2
 =\gamma s^2
\end{align*}
for every disjoint $s$-vertex pair $A,B$.

We may therefore apply Theorem~\ref{thm:GW} to the host graph $H_i$ with the parameters above. It yields a spanning triangle-free graph $G_i\subseteq H_i$ such that $e_{G_i}(A,B)=(1\pm\delta)\rho e_{H_i}(A,B)$ for every pair of disjoint $s$-vertex sets $A,B$. Since $H_{i+1}=H_i\setminus G_i$, this gives
\[
 \bigl(1-(1+\delta)\rho\bigr)e_{H_i}(A,B)
 \le e_{H_{i+1}}(A,B)
 \le \bigl(1-(1-\delta)\rho\bigr)e_{H_i}(A,B),
\]
which proves the invariant for $i+1$. The $G_i$ are edge-disjoint by construction.

\noindent\emph{Uniform feedback-arc robustness after random orientation.}
For every $i<I$ and every pair of disjoint $s$-vertex sets $A,B$, the preceding lower bound on $H_i$ together with $e_{G_i}(A,B)=(1\pm\delta)\rho e_{H_i}(A,B)$ imply
\[
    e_{G_i}(A,B)\ge(1-\delta)\rho\gamma s^2.
\]
Let $q=(1-\delta)\gamma\rho$.
Since $k-1\ge s$, the last condition on $L$ gives
\[
    \kappa q(k-1)
    \ge\kappa(1-\delta)\gamma\rho s
    \ge8\log_2 n.
\]
Theorem~\ref{thm:transfer} applies to every $G_i$. Orient the edges of all $G_i$ independently and uniformly, obtaining $D_i$. A union bound over $i$ yields
\[
 \Pp\!\left(\exists i<I,\ \exists U,\ |U|\ge k:\ \fas(D_i[U])\le\left(\tfrac12-\zeta\right)e(G_i[U])\right)\le2I n^{-3k}.
\]
Now $I=O_\eps(\sqrt{n/\log n})$ and $k\ge2L\sqrt{n\log n}$. Hence, for some $c=c(\eps,\xi,C_0,\zeta)>0$ and all sufficiently large $n$,
\[
 2I n^{-3k}
 \le \exp\!\left[-c\sqrt n\,(\log n)^{3/2}\right].
\]
Also $k\le3L\sqrt{n\log n}$ for large $n$. This is the probability bound of Theorem~\ref{thm:packing}, with $k_n=k$ and $C=3L$.

On the good event, no $D_i$ contains an acyclic set of order $k$. Therefore every acyclic set in $D_i[W]$ has order at most $k-1$, and the hereditary dichromatic bound follows from the pigeonhole principle.

\noindent\emph{Edge coverage.}
From the upper bound in the invariant and the definition of $I$,
\begin{align*}
 e_{H_I}(A,B)
 &\le\bigl(1-(1-\delta)\rho\bigr)^I e_H(A,B)\\
 &\le\exp\bigl(-(1-\delta)\rho I\bigr)e_H(A,B)\\
 &\le\eps e_H(A,B)
\end{align*}
for every pair of disjoint $s$-vertex sets $A,B$. Let
\[
    F=H\setminus H_I=\bigcup_{i=0}^{I-1}G_i.
\]
Then
\[
    e_F(A,B)\ge(1-\eps)e_H(A,B)
\]
for all such pairs. Lemma~\ref{lem:edge-double-count} yields
$e(F)\ge(1-\eps)e(H)$.
Because the $G_i$ are edge-disjoint,
$e(F)=\sum_{i=0}^{I-1}e(G_i)$,
which proves the coverage bound $\sum_i e(G_i)\ge(1-\eps)e(H)$ and completes the proof.
\end{proof}

\begin{proof}[Proof of Corollary~\ref{cor:complete-packing}]
Apply Theorem~\ref{thm:packing} with $H=K_n$, $\xi=1$, and any fixed $C_0>0$. The lower-uniformity assumption is automatic, and all conclusions follow directly.
\end{proof}

\begin{proof}[Proof of Corollary~\ref{cor:random-host}]
Let $t=\left\lceil\sqrt{n\log n}\right\rceil$.
Fix disjoint $t$-vertex sets $A,B$. In $H\sim G(n,p)$, the random variable $e_H(A,B)$ has distribution $\operatorname{Bin}(t^2,p)$. A standard Chernoff bound gives
\[
    \Pp\!\left(e_H(A,B)<\frac p2t^2\right)
    \le\exp\!\left(-\frac p8t^2\right).
\]
There are at most $3^n$ ordered pairs of disjoint vertex sets. Hence
\[
 \Pp\!\left(
   \exists A,B,\ |A|=|B|=t:\
   e_H(A,B)<\frac p2t^2
 \right)
 \le3^n\exp\!\left(-\frac p8t^2\right)=o(1),
\]
because $t^2\ge n\log n$. Thus $H$ is $(p/2,t)$-lower-uniform with probability $1-o(1)$. Also, by the ordinary Chernoff bound,
$e(H)=(1+o(1))p\binom n2$
with probability $1-o(1)$.

On this event, apply Theorem~\ref{thm:packing} with $\xi=p/2$ and $C_0=1$. This gives the deterministic packing inside the realised host $H$, and the conditional probability estimate for independent random orientations. Combining the two high-probability events proves the corollary.
\end{proof}

The complete-graph packing can also be viewed inside one random tournament.

\begin{corollary}\label{cor:random-tournament}
For every $\eps>0$ and $\zeta\in(0,1/2)$ there are constants $C,R,c>0$ such that, for all sufficiently large $n$, there is a deterministic edge-disjoint family
\[
    G_0,\ldots,G_{r-1}\subseteq K_n,
    \quad
    r=\left\lceil R\sqrt{\frac n{\log n}}\right\rceil,
\]
of spanning triangle-free graphs covering at least $(1-\eps)\binom n2$ edges, and an integer $k_n\le C\sqrt{n\log n}$, with the following property. If $T_n$ is a uniformly random tournament on the same vertex set and $D_i$ is the restriction of $T_n$ to $G_i$, then
\[
 \Pp\!\left(\exists i<m,\ \exists U,\ |U|\ge k_n:\ \fas(D_i[U])\le\left(\tfrac12-\zeta\right)e(G_i[U])\right)\le\exp\!\left[-c\sqrt n\,(\log n)^{3/2}\right].
\]
In particular, with high probability every $i<m$ and every $W$ with $|W|\ge k_n$ satisfy $\dich(D_i[W])\ge\frac{|W|}{k_n-1}$.
\end{corollary}

\begin{proof}
Apply Theorem~\ref{thm:packing} to $H=K_n$ with $\xi=1$ and any fixed $C_0>0$, and retain the resulting deterministic undirected packing. A uniformly random tournament is obtained by orienting all edges of $K_n$ independently and uniformly. Since the $G_i$ are edge-disjoint, the restrictions to their edge sets are independent uniform orientations. The conclusion is therefore exactly the probability estimate in Theorem~\ref{thm:packing}.
\end{proof}

\section{Concluding remarks}

The feedback-arc conclusion is first-order best possible: Proposition~\ref{prop:fas-orders} gives $\fas(D[U])\le e(G[U])/2$ for every orientation, while Theorems~\ref{thm:single} and~\ref{thm:packing} achieve $(1/2-\zeta)e(G[U])$ simultaneously in every sufficiently large induced subdigraph. The proof uses deliberately wasteful constants. The scale constant supplied by the Guo--Warnke theorem is enlarged until the entropy gap
$1-\hbin\!\left(\frac12-\zeta\right)$ dominates the choices of a vertex set and an ordering.

The Guo--Warnke construction is algorithmic, and the subsequent random orientation is immediate. Consequently, the graphs and orientations in our main theorems can be sampled by a randomized polynomial-time procedure with the stated success probability, using the algorithmic version of their theorem.

We conclude with three natural questions.

\begin{question}
Does a uniformly random orientation of the terminal graph of the triangle-free process satisfy the same uniform feedback-arc conclusion as Theorem~\ref{thm:single}?
\end{question}

\begin{question}
What are the optimal leading constants in $\amin(n)$ and $\tmax(n)$? In particular, can one match the inverse-Ramsey constant in the lower bound for $\amin(n)$?
\end{question}

\begin{question}
Can $K_n$ be completely decomposed, up to divisibility obstructions, into spanning oriented triangle-free graphs for which every sufficiently large induced subdigraph has feedback arc set number $(1/2-o(1))$ times its number of arcs?
\end{question}

The last question is a robust oriented strengthening of the complete-decomposition problem posed by Guo and Warnke for nearly optimal Ramsey graphs.


\begin{thebibliography}{99}

\bibitem{AboulkerHPS}
P.~Aboulker, F.~Havet, F.~Pirot, and J.~Schabanel,
\newblock Minimum acyclic number and maximum dichromatic number of oriented triangle-free graphs of a given order,
\newblock \emph{Electron. J. Combin.} \textbf{32} (2025), no.~4, Paper~P4.27.

\bibitem{AKS}
M.~Ajtai, J.~Koml\'os, and E.~Szemer\'edi,
\newblock A note on Ramsey numbers,
\newblock \emph{J. Combin. Theory Ser. A} \textbf{29} (1980), 354--360.

\bibitem{BohmanKeevash}
T.~Bohman and P.~Keevash,
\newblock Dynamic concentration of the triangle-free process,
\newblock \emph{Random Structures Algorithms} \textbf{58} (2021), 221--293.

\bibitem{FizGriffithsMorris}
G.~Fiz Pontiveros, S.~Griffiths, and R.~Morris,
\newblock The triangle-free process and the Ramsey number $R(3,k)$,
\newblock \emph{Mem. Amer. Math. Soc.} \textbf{263} (2020), no.~1274.

\bibitem{GuoWarnke}
H.~Guo and L.~Warnke,
\newblock Packing nearly optimal Ramsey $R(3,t)$ graphs,
\newblock \emph{Combinatorica} \textbf{40} (2020), 63--103.

\bibitem{Kim}
J.~H.~Kim,
\newblock The Ramsey number $R(3,t)$ has order of magnitude $t^2/\log t$,
\newblock \emph{Random Structures Algorithms} \textbf{7} (1995), 173--207.

\bibitem{NeumannLara}
V.~Neumann-Lara,
\newblock The dichromatic number of a digraph,
\newblock \emph{J. Combin. Theory Ser. B} \textbf{33} (1982), 265--270.

\bibitem{Shearer}
J.~B.~Shearer,
\newblock A note on the independence number of triangle-free graphs,
\newblock \emph{Discrete Math.} \textbf{46} (1983), 83--87.

\end{thebibliography}
\end{document}